\documentclass[11pt]{article}
\usepackage{amsmath,amssymb,fullpage,graphicx}
\usepackage[authoryear]{natbib}
\usepackage{amsthm}
\usepackage{algorithm}
\usepackage{algpseudocode}
\usepackage{appendix}
\usepackage{titlesec}
\usepackage[affil-it]{authblk}  
\usepackage[hidelinks]{hyperref}
\usepackage{cleveref}

\bibpunct{(}{)}{;}{a}{,}{,}

\newtheorem{theorem}{Theorem}
\newtheorem{lemma}[theorem]{Lemma}

\newtheorem{corollary}[theorem]{Corollary}
\newtheorem{proposition}[theorem]{Proposition}

\theoremstyle{definition}
\newtheorem{definition}[theorem]{Definition}

\newcommand{\figref}[1]{Figure~\ref{#1}}

\usepackage{color}

\algnewcommand\algorithmicinput{\textbf{INPUT:}}
\algnewcommand\INPUT{\item[\algorithmicinput]}
\algnewcommand\algorithmicoutput{\textbf{OUTPUT:}}
\algnewcommand\OUTPUT{\item[\algorithmicoutput]}

\newcommand{\R}{\mbox{$\mathbb{R}$}}
\newcommand{\X}{\mbox{$\mathbb{X}$}}
\newcommand{\Z}{\mbox{$\mathbb{Z}$}}

\newenvironment{enum}{
\begin{enumerate}
  \setlength{\itemsep}{1pt}
  \setlength{\parskip}{0pt}
  \setlength{\parsep}{0pt}
}{\end{enumerate}}

\titlespacing{\section}{0pt}{11pt}{0pt}
\titlespacing{\subsection}{0pt}{8pt}{-4pt}
\titlespacing{\subsubsection}{0pt}{10pt}{-4pt}

\usepackage[font=small,skip=0pt]{caption}
\usepackage{fouriernc}

\usepackage{etoolbox}
\newtoggle{journalv}
\togglefalse{journalv} 

\newcommand{\dgm}{D} 
\newcommand{\dgmspace}{\mathcal{D}_T} 
\newcommand{\lscape}{\lambda} 
\newcommand{\lscapespace}{\mathcal{L}_T} 

\let\hat\widehat
\let\tilde\widetilde

\title{\Large Stochastic Convergence of Persistence Landscapes and Silhouettes}

\author[1]{Fr\'ed\'eric Chazal}
\author[2]{Brittany Terese Fasy}
\author[3]{Fabrizio Lecci}
\author[3]{\\Alessandro Rinaldo}
\author[3]{Larry Wasserman}

\affil[1]{INRIA Saclay}
\affil[2]{Computer Science Department, Tulane University}
\affil[3]{Department of Statistics, Carnegie Mellon University}

\date{topstat@stat.cmu.edu \\ \text{} \\ \text{} \\
November 27, 2013
}

\begin{document}

\maketitle

\begin{abstract}
\noindent
Persistent homology is a widely used tool in Topological Data Analysis that
encodes multiscale topological information
as a multi-set of points in the plane called a persistence diagram.
It is difficult to apply statistical theory directly
to a random sample of diagrams.
Instead, we can summarize the persistent homology
with the persistence landscape, introduced by
Bubenik, which converts a diagram into a
well-behaved real-valued function.  We investigate
the statistical properties of landscapes, such as
weak convergence of the average landscapes and
convergence of the bootstrap.  In addition, we
introduce an alternate functional summary of persistent
homology, which we call the silhouette, and derive an analogous statistical
theory.
\end{abstract}

\clearpage
\section{Introduction}
\vskip-5pt

Often, data can be represented as point clouds that carry
specific topological and geometric structures.  Identifying,
extracting, and exploiting these underlying geometric structures has
become a problem of fundamental importance for data analysis and
statistical learning. With the emergence of new geometric inference
and algebraic topology tools, computational topology has recently seen
an important development toward data analysis, giving birth to the
field of Topological Data Analysis, whose aim is to infer
relevant, multiscale, qualitative, and quantitative topological
structures directly from the data.  

Persistent homology
(\cite{edels2002topological,zc-cph-05}) is a fundamental tool
for providing multi-scale homology descriptors of data. More precisely, it
provides a framework and efficient algorithms to quantify the
evolution of the topology of a family of nested topological spaces,
$\{\X(t) \}_{t \in \R}$, built on top of the data and indexed by a set
of real numbers -- that can be seen as scale parameters -- such that
$\X(t) \subseteq \X(s)$ for all $t \leq s$.  At the homology 
level\footnote{We consider here homology with coefficient in a given field,
so the homology groups are vector spaces.}, such a filtration
induces a family $\{H(\X(t)) \}_{t \in \R}$ of homology groups and the
inclusions $\X(t) \hookrightarrow \X(s)$ induce a family of
homomorphisms $H(\X(t)) \to H(\X(s))$, $t \leq s$, which is known as the
persistence module associated to the filtration.  When the rank of all
the homomorphisms $H(\X(t)) \to H(\X(s))$, $t < s$, are finite the
module is said to be q-tame (\cite{chazal2012structure}) and it can be
summarized as a set of real intervals $(b_i,d_i)$ representing
homological features that appear in the filtration at $t= b_i$ and
disappear at $t= d_i$. Such a set of intervals can be represented as a
multi-set of points in the real plane and is then called a persistence
diagrams.
Thanks to their stability properties
(\cite{stability,chazal2012structure}), persistence diagrams provide
relevant multi-scale topological information about the data. 

In a more
statistical framework, when several data sets are randomly generated
or are coming from repeated experiments, one often has to deal with
not only one persistence diagrams but with a whole distribution of
persistence diagrams. Unfortunately, since the space of persistence diagrams
is a general metric space, analyzing and quantifying the
statistical properties of such a distribution turns out to be
particularly difficult.

A few attempts have been made towards a statistical analysis of
distributions of persistence diagrams.  For example, the concentration
and convergence properties of persistence diagrams obtained from point
cloud randomly sampled on manifolds and from more general compact metric
spaces are studied in \cite{us2013, cglm-orcpd-13}. Considering
general distributions of persistence diagrams,
\cite{turner2012frechet} have suggested using the Fr\'echet average of
the diagrams $D_1,\ldots, D_n$.  Unfortunately, the Fr\'echet average
is unstable and not even unique.  A solution that uses a probabilistic
approach to define a unique Fr\'echet average can be found in
\cite{munch2013probabilistic}, but its computation remains practically
prohibitive.

In this paper, we also consider general distributions of persistence diagrams but we build on a completely different approach, 
proposed in
\cite{bubenik2012statistical}, consisting of encoding persistence
diagrams as a collection of real-valued one-Lipschitz functions 
that are called
persistence landscapes; see Section \ref{sec:landscapes}.
The advantage of landscapes --- and, more generally, of
any function-valued summaries of persistent homology ---
is that we can analyze them using
existing techniques and theories from nonparametric statistics.

We have in mind two scenarios
where multiple persistence diagrams arise:

{\bf Scenario 1:}
We have a random sample of compact sets
$K_1,\ldots,K_n$ drawn
from a probability distribution on the space
of compact sets.
Each set $K_i$ gives rise to a persistence diagram which in turn yields
a persistence landscape function $\lambda_i$.
An analogous sampling scenario is the one where we observe a sample of $n$
random Morse
functions $f_1,\ldots,f_n$ from a common probability distribution. Each such
function $f_i$ induces a persistent
diagram built from its sub-level set filtration, which can again
be encoded by a landscape $\lambda_i$. 
The goal is to use the observed landscapes
$\lambda_1,\ldots, \lambda_n$ to infer 
the mean landscape
\mbox{$\mu = \mathbb{E}(\lambda_i)$.}

{\bf Scenario 2:}
We have a very large dataset
with $N$ points.
There is a diagram $D$ and landscape $\lambda$
corresponding to some filtration built on the data.
When $N$ is large, computing $D$ is prohibitive.
Instead, we draw $n$ subsamples, each of size $m$.
We compute a diagram and landscape for each subsample
yielding landscapes
$\lambda_1,\ldots,\lambda_n$.
(Assuming $m$ is much smaller than $N$, these subsamples
are essentially independent and identically distributed.)
Then we are interested in estimating
\mbox{$\mu =\mathbb{E}(\lambda_i)$},
which can be regarded as an approximation of $\lambda$.
Two questions arise:
how far are the $\lambda_i$'s from their mean $\mu$
and how far is $\mu$ from $\lambda$.
We focus on the first question in this paper.

In both sampling scenarios, we 
study the statistical behavior as the number of persistence
diagrams $n$ grows. We will then analyze the stochastic limiting
behavior of the average landscape, as well as the speed of convergence to such 
limit. Specifically, the contributions of this papers are as follows:\nopagebreak
\begin{enum}
\item We show that the average persistence landscape converges weakly to a Gaussian process and we
find the rate of convergence of that process.
\item We show that a statistical procedure
known as the bootstrap leads to valid confidence bands
for the average landscape. We provide an algorithm to compute confidence bands and illustrate it on a few real and simulated examples. 
\item We define a new functional summary of persistent homology,
which we call the \emph{silhouette}. 
\end{enum}

As the proofs are rather technical, we defer the interested reader to the appendices.

\section{Persistence Diagrams and Landscapes} \label{sec:landscapes}
\vskip-5pt

Formally, a (finite)
persistence diagram is a set of real intervals $\{ (b_i,d_i)
\}_{i \in I}$ where $I$ is a finite set.  We represent a persistence
diagram as the finite multiset of points $D = \left\{
  (\frac{b_i+d_i}{2},\frac{d_i-b_i}{2})\right\}_{i \in I}$.  Given a positive real
number $T$, we say that $D$ is $T$-bounded if for each point~$(x,y) =
\left(\frac{d+b}{2}, \frac{d-b}{2}\right) \in D$, we have $0 \leq b \leq d \leq
T$. We denote by $\dgmspace$ the space of all positive, finite,
$T$-bounded persistence diagrams.

A persistence landscape, introduced in \cite{bubenik2012statistical},
is a set of continuous, piecewise linear 
functions~\mbox{$\lscape \colon  \Z^{+} \times \R \to \R$} which provides an
encoding of a persistence diagram.  
To define the landscape, consider the set of functions created by tenting
each persistence point $p = (x,y)= \left( \frac{b+d}{2}, \frac{d-b}{2}  \right)
\in \dgm$
to the base line $x=0$ as with the following function:
\begin{equation}\label{eq:triangle}
 \Lambda_p(t) ~=~
 \begin{cases}
  t-x+y & t \in [x-y, x] \\
  x+y-t & t \in (x,  x+y] \\
  0 & \text{otherwise}
 \end{cases}
 ~=~
 \begin{cases}
  t-b & t \in [b, \frac{b+d}{2}] \\
  d-t & t \in (\frac{b+d}{2}, d] \\
  0 & \text{otherwise}.
 \end{cases}
\end{equation}
Notice that $p$ is itself on the graph of $\Lambda_p(t)$.
We obtain
an arrangement of curves by overlaying the graphs of 
the functions~\mbox{$\{ \Lambda_p \}_{p \in \dgm}$}; see~\figref{fig:landscape}.  

\begin{figure}
\centering
\includegraphics[height=1.7in]{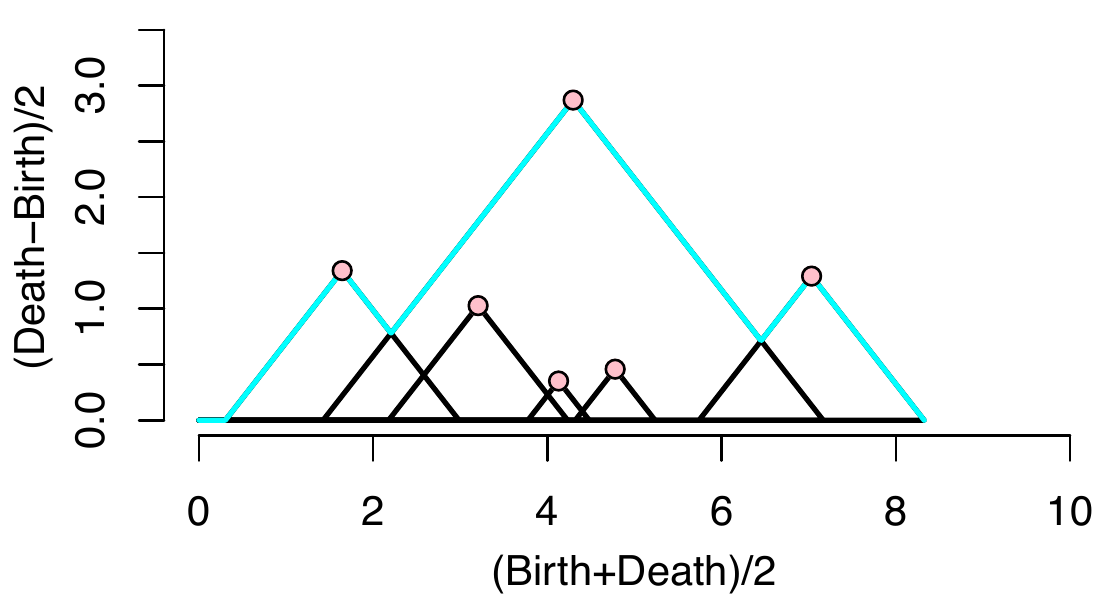}
\caption{The pink circles are the points in a persistence diagram $D$.
          Each point $p$ corresponds to a function $\Lambda_p$ given in \eqref{eq:triangle},
          and the landscape $\lscape(k, \cdot)$ is the $k$-th largest
          of the arrangement of the graphs of $\{ \Lambda_p \}$.
          In particular, the cyan curve is the landscape $\lscape(1,\cdot)$.}
\label{fig:landscape}
\end{figure}

The persistence landscape of $\dgm$ is just a  summary of this arrangement.
Formally, the persistence landscape of $D$ is the collection of functions
\begin{equation}\label{eq:landscape}
 \lscape_{\dgm}(k,t) = \underset{p \in \dgm}{\text{kmax}} ~ \Lambda_p(t), \quad
 t \in [0,T], k \in \mathbb{N},
\end{equation}
where kmax is the $k$th largest value in the set; in particular,
$1$max is the usual maximum~function. We set $\lscape_{\dgm}(k,t) = 0$ if the set $\{ \Lambda_p(t), p \in \dgm\}$ contains less than
$k$ points.
From the definition of persistence landscape, we immediately observe that
$\lscape_{\dgm}(k,\cdot)$ is one-Lipschitz,  since $\Lambda_p$ is one-Lipschitz.
We denote by $\lscapespace$ the space of persistence landscapes corresponding to $\dgmspace$.\\
For ease of exposition, in this paper we only focus on the case $k=1$,
and set $\lscape(t) = \lscape_{\dgm}(1,t)$.  However,
the results we present hold for~$k > 1$.


\section{Weak Convergence of Landscapes}
\label{ss::weak}
\vskip-5pt

Let $P$ be a probability distribution on $\lscapespace$, and
let $\lscape_1, \ldots, \lscape_n \sim P$.
We define the mean landscape as
$$
\mu(t) = \mathbb{E}[ \lscape_{i}(t)], \quad t \in [0,T].
$$
The mean landscape is an unknown function that
we would like to estimate.
We estimate $\mu$ with the sample average
$$
\overline{\lscape}_n(t) = \frac{1}{n}\sum_{i=1}^n \lscape_{i}(t), \quad t \in [0,T].
$$
Note that since 
$\mathbb{E}(\overline{\lscape}_n(t))=\mu(t)$,
we have that $\overline{\lscape}_n$ is a point-wise unbiased estimator of the unknown
function $\mu$.
Our goal is then quantify how close
the resulting estimate is to the function $\mu$.
To do so, we first need to explore the statistical properties
of $\overline{\lscape}_n$.
\cite{bubenik2012statistical} showed that $\overline{\lscape}_n$ converges pointwise to $\mu$ 
and that the pointwise Central Limit Theorem holds.
In this section we extend these results, proving the uniform convergence of the average landscape. 
In particular, we show that the process
\begin{equation}\label{eq:landscapeprocess}
  \Big\{ \sqrt{n}\left(\overline{\lscape}_n(t) - \mu(t)\right) \Big\}_{t \in [0,T]}
\end{equation}
converges weakly (see below) to a Gaussian process on $[0,T]$
and we establish the rate of convergence.

Let
\begin{equation}
\label{eq::classF} 
{\cal F} = \{ f_t  \}_{0\leq t\leq T}
\end{equation}
where
$f_t:{\lscapespace}\to \mathbb{R}$ is defined by
$f_t(\lscape) = \lscape(t)$.
Writing $P(f) = \int f dP$ and letting $P_n$ be the empirical measure
that puts mass $1/n$ at each $\lscape_i$,
we can and will regard \eqref{eq:landscapeprocess} as an empirical process indexed by $f_t \in
\mathcal{F}$. Thus, for $t \in [0,T]$, we will write
\begin{equation}
\label{eq::empirical}
\mathbb{G}_n(t) ~=~  \mathbb{G}_n(f_t):=\sqrt{n}\left(\overline{\lscape}_n(t) - \mu(t)\right) 
  ~=~ \frac{1}{\sqrt{n}}\sum_{i=1}^n \left(f_t(\lscape_i) - \mu(t)\right) ~=~
  \sqrt{n}(P_n - P)(f_t)
\end{equation}
We note that
the function
$F(\lambda) = T/2$ is a measurable envelope for ${\cal F}$.

A Brownian bridge is a mean zero Gaussian process on
the set of bounded functions from ${\cal F}$ to $\mathbb{R}$
such that the covariance between any pair
$f_1,f_2\in {\cal F}$ has the form
$\int f_1(u) f_2(u) dP(u) - \int f_1(u) dP(u) \int  f_2(u) dP(u)$.
A sequence of random objects $X_n$
converges weakly to $X$, written $X_n\rightsquigarrow X$,
if $\mathbb{E}^*(f(X_n)) \to \mathbb{E}(f(X))$
for every bounded continuous function $f$.
(The symbol $\mathbb{E}^*$ is an outer expectation, which is used for technical reasons;
the reader can think of this as an expectation.)
\begin{figure}
 \centering
 \includegraphics[height=1.8in]{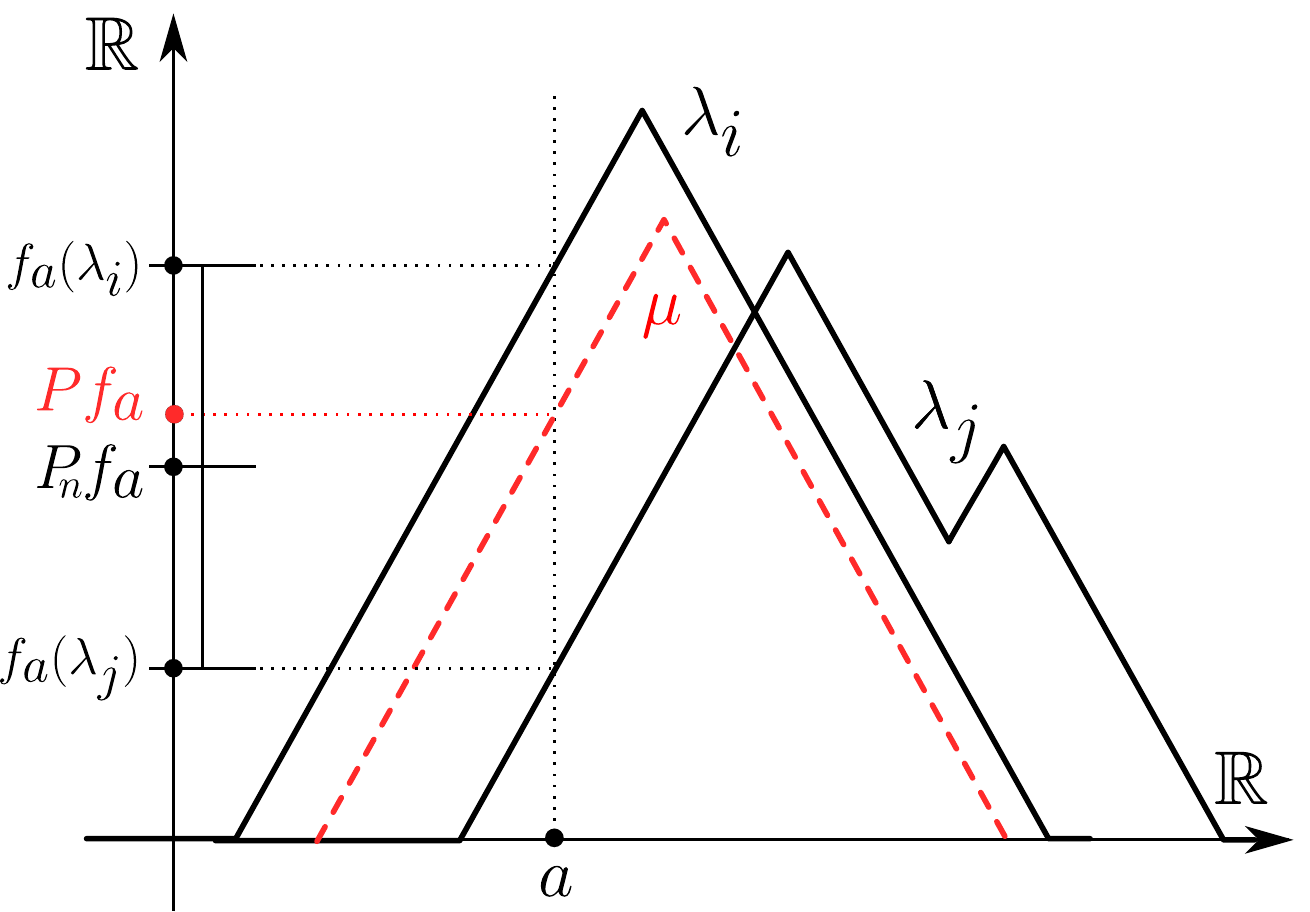}
 \caption{We illustrate the empirical process $\mathbb{G}_n(f_t)$.  Given a set of landscapes $\{\lambda_i\}_{1 \leq i \leq n}$,
 each real-value $a$ corresponds to a function $f_a \colon \lscapespace \to \R$ defined by $f_a(\lambda_i) = \lambda_i(a)$.
 ${P}_nf_a$ is then the average over all sampled landscapes.  If $\mu$ is the true mean landscape,
 then $Pf_a = \mu(a)$ and $\mathbb{G}_n(f_a)$ is the normalized difference $\sqrt{n}(\mathbb{P}_nf_a - Pf_a)$.}
 \label{fig:gaussianprocess}
\end{figure}
Thus, we arrive at the following~theorem:

\begin{theorem}[Weak Convergence of Landscapes, Theorem 2.4 in \cite{usBootstrap2013}]
\label{th::weak}
Let $\mathbb{G}$ be a Brownian bridge
with covariance function
$
\kappa(t,s) = \int f_t(\lambda)f_s(\lambda) dP(\lambda) - \int f_t(\lambda)
dP(\lambda) \int f_s(\lambda) dP(\lambda)$, 
for $t,s \in [0,T].$ 
Then
$\mathbb{G}_n \rightsquigarrow \mathbb{G}.$
\end{theorem}

Next, we describe the rate of convergence 
of the maximum of the normalized empirical process $\mathbb{G}_n$ to the 
maximum of the limiting distribution $\mathbb{G}$.
The maximum is relevant for statistical inference
as we shall see in the next section.

For each $t \in [0,T]$, let $\sigma(t)$ be the standard deviation of $\sqrt{n}\, \overline{\lscape}_n(t)
$, i.e.
\begin{equation}
\label{eq::sigma}
\sigma(t)= \sqrt{n \text{Var}(\overline{\lscape}_n(t))}= \sqrt{\text{Var}(f_t(\lscape_1))}.
\end{equation}

\begin{theorem}[Uniform CLT]
\label{th::CLT}
Suppose that $\sigma(t)>c>0$ in an interval $[t_* \, , t^*] \subset [0,T]$, for some constant  $c$. 
Then there exists a random variable $W\stackrel{d}{=} \sup_{t \in [t_* \, , t^*]} |\mathbb{G}(f_t) |$ such that
$$
\sup_{z\in \R} \left| \mathbb{P}\left(\sup_{t \in [t_* \, , t^*]} |\mathbb{G}_n(t)| \leq z
\right) - \mathbb{P}\left(W \leq z \right) \right| = 
O \left( \frac{(\log n)^{7/8} }{n^{1/8}}\right).
$$
\end{theorem}

{\bf Remarks:}
The assumption in \Cref{th::CLT} that the standard deviation function $\sigma$
is positive over a subinterval of $[0,T]$ can be replaced with the weaker
assumption of positivity of $\sigma$ over a finite collection of sub-intervals
without changing the result. We
have stated the theorem in this simplified form for ease of readability.
Furthermore, it may be possible to improve the term $n^{-1/8}$ in the rate
using what is known as a ``Hungarian embedding''
(see Chapter 19 of \cite{van2000asymptotic}).
We do not pursue this point further, however.

\section{The Bootstrap for Landscapes}
\label{ss::bootstrap}
\vskip-5pt
Recall that our goal is to use the observed landscapes
$(\lscape_{1},\ldots, \lscape_{n})$
to make inferences about
$\mu(t) = \mathbb{E}[ \lscape_{i}(t)]$, where $0 \leq t \leq T$.
Specifically, in this paper we will seek to construct an asymptotic {\it confidence band} for $\mu$. A pair of functions
$\ell_n,u_n \colon \R \to \R$ is an asymptotic $(1-\alpha)$ confidence band for
$\mu$ if, as $n \rightarrow \infty$,
\begin{equation}
\mathbb{P}\Bigl( \ell_n(t) \leq \mu(t) \leq u_n(t)\ {\rm for\ all\ }t\Bigr) \geq 1-\alpha - O(r_n),
\end{equation}
where $r_n = o(1)$. Confidence bands are valuable tools for statistical
inference, as they allow to quantify and visualize the uncertainty about the
mean persistence landscape function $\mu$ and to screen out topological noise.
 
Below we will describe an algorithm for constructing the funcions $\ell_n$ and
$u_n$ from
the sample of landscapes $\lscape_1^n := (\lscape_1,\ldots,\lscape_n)$,  will prove that it yields
an asymptotic $(1-\alpha)$-confidence band for the unknown mean landscape
function $\mu$ and determine its rate $r_n$.
Our algorithm relies on the use of the {\it bootstrap}, a simulation-based statistical method for constructing
confidence set under minimal assumptions on the data generating distribution $P$; see \cite{efron1979bootstrap, efron1993introduction,van2000asymptotic}.
There are several different versions of the bootstrap.
This paper uses
the {\em multiplier bootstrap}.

Let $\xi_1^n=(\xi_1,\ldots, \xi_n)$ where $\xi_i \sim N(0,1)$ 
(Gaussian random variables wit mean 0 and variance 1)
for all $i$ and
define the multiplier bootstrap process
\begin{equation}
\label{eq::empiricalBoot}
\tilde{\mathbb{G}}_n(f_t) = \tilde{\mathbb{G}}_n(\lscape_1^n, \xi_1^n)(f_t):= 
\frac{1}{\sqrt{n}} \sum_{i=1}^n \xi_i \left( f_t(\lscape_i)- {\overline \lscape}_n(t) \right) \, , \; t \in [0,T].
\end{equation}
Let $\tilde Z(\alpha)$ be the unique value such that
\begin{equation}\label{eq::zalpha}
\mathbb{P}\left(
\sup_t 
\Bigl| \tilde{\mathbb{G}}_n(f_t) \Bigr| > \tilde Z(\alpha)\ \Biggm| \ \lscape_1, \ldots, \lscape_n\right)=
\alpha.
\end{equation}
Note that the only random quantities in this definition are
$\xi_1,\ldots, \xi_n \sim N(0,1)$.
Hence, $\tilde Z(\alpha)$ can be approximated by Monte Carlo simulation.
Let $\tilde \theta=\sup_{t \in [0,T]} |\tilde{\mathbb{G}}_n(f_t)|$ 
be from a bootstrap sample. Repeat the bootstrap B times, yielding values $\tilde \theta_1, \dots, \tilde \theta_B$. Let
\begin{equation}\label{eq::sim}
\tilde Z(\alpha)= \inf \left\{z: \frac{1}{B} \sum_{j=1}^B I(\tilde \theta_j>z) \leq \alpha \right\}.
\end{equation}
We may take $B$ as large as we like so the Monte Carlo error arbitrarily small.
Thus, when using bootstrap methods,
one ignores the error in approximating $\tilde{Z}(\alpha)$ as defined in
(\ref{eq::zalpha}) with its simulation approximation as defined in
(\ref{eq::sim}).
The multiplier bootstrap confidence band is
$\{ (\ell_n(t),u_n(t)):\ 0\leq t \leq T\}$, where
\begin{equation}
\ell_n(t)= \overline{\lscape}_n(t) - \frac{\tilde Z(\alpha)}{\sqrt{n}},\ \ \ 
u_n(t) = \overline{\lscape}_n(t) + \frac{\tilde Z(\alpha)}{\sqrt{n}}.
\end{equation}

The steps of the algorithm are given in Algorithm \ref{alg::bootstrap}.

\begin{algorithm}
  \caption{The multipler bootstrap algorithm.}\label{alg::bootstrap}
  \begin{algorithmic}[1]
  \INPUT Landscapes ${\lscape}_1,\ldots, {\lscape}_n$; confidence level $1-\alpha$; number of bootstrap samples $B$
  \OUTPUT confidence functions $\ell_n, u_n \colon \R \to \R$
    \State Compute the average $\overline{\lscape}_n(t) = \frac{1}{n}\sum_{i=1}^n {\lscape}_i(t)$
    \For{$j=1 \text{ to } B$} 
      \State Generate $\xi_1,\ldots, \xi_n \sim N(0,1)$
      \State Set  $\tilde\theta_j = \sup_t n^{-1/2} |\sum_{i=1}^n \xi_i\ ({\lscape}_i(t) - \overline{\lscape}_n(t))|$
    \EndFor  
      \State Define $\tilde Z(\alpha) = \inf\bigl\{ z:\ \frac{1}{B}\sum_{j=1}^B I( \tilde\theta_j > z) \leq \alpha\bigr\}$
    \State Set
    $\ell_n(t) = \overline{\lscape}_n(t) - \frac{\tilde Z(\alpha)}{\sqrt{n}}$ and
    $u_n(t) = \overline{\lscape}_n(t) + \frac{\tilde Z(\alpha)}{\sqrt{n}}$
    \State \textbf{return} $\ell_n(t), u_n(t)$
  \end{algorithmic}
\end{algorithm}

The accuracy of the coverage of the confidence band and the width of
the band are described in the next result, which follows from Theorem
\ref{th::CLT} and the analogous result for the multiplier bootstrap
process, stated in Proposition \ref{prop::bootstrap} in Appendix
\ref{sec::appendixB}.

\begin{theorem}[Uniform Band]
\label{th::band}
Suppose that $\sigma(t)>c>0$ in an interval $[t_* \, , t^*] \subset [0,T]$, for some constant  $c$. Then
\begin{equation}
\mathbb{P}\Bigl( \ell_n(t) \leq \mu(t) \leq u_n(t)\ {\rm for\ all\ }t \in [t_* \, , t^*]\Bigr) \geq 1-\alpha - 
O \left( \frac{(\log n)^{7/8}}{n^{1/8}} \right).
\end{equation}
Also,
$\sup_t \left(u_n(t) - \ell_n(t)\right) = O_P\left( \sqrt{\frac{1}{n}}\right)$.
\end{theorem}

The confidence band above has a constant width; that is, the width is the same for all $t$.
However, the empirical estimate $\overline{\lscape}(t)$ might be a more accurate estimator of~$\mu(t)$
for some~$t$ than others. This suggests that we may construct a more refined
confidence band whose width varies with $t$.
Hence, we construct an {\em adaptive confidence band} that has variable width.
Consider the standard deviation function $\sigma$, defined in \eqref{eq::sigma}, and its estimate
\begin{equation}
\label{eq::sigmaHat}
\hat \sigma_n(t):= \sqrt{\frac{1}{n} \sum_{i=1}^n [f_t(\lscape_i)]^2 - [
    \overline{\lscape}_n(t)) ]^2}, \quad t \in [0,T].
\end{equation}
Set $T_\sigma = \{t \in [0,T] \colon \sigma(t) > 0 \}$ and define the standardized empirical process 
\begin{equation}
\mathbb{H}_n(f_t) := \mathbb{H}_n(\lscape_1^n)(f_t):= 
\frac{1}{\sqrt{n}} \sum_{i=1}^n  \frac{ f_t(\lscape_i)- \mu(t) }{\sigma(t)},
\quad t \in T_\sigma
\end{equation}
and, for $\xi_1, \dots, \xi_n \sim N(0,1)$, define its multiplier bootstrap version
\begin{equation}
\label{eq::empiricalBootSigmaHat}
\hat{\mathbb{H}}_n(f_t) := \hat{\mathbb{H}}_n(\lscape_1^n, \xi_1^n)(f_t):= 
\frac{1}{\sqrt{n}} \sum_{i=1}^n \xi_i \frac{ f_t(\lscape_i)-
    \overline{\lscape}_n(t) }{\hat\sigma_n(t)}, \quad t \in T_\sigma.
\end{equation}
Just like in the construction of uniform bands, let $\hat Q(\alpha)$ be such
that
\begin{equation}
\mathbb{P}\left(
\sup_{t } 
\Bigl| \hat{\mathbb{H}}_n(\lambda_1^n, \xi_1^n)(f_t) \Bigr| > \hat Q(\alpha)\ \Biggm| \ \lambda_1,\ldots, \lambda_n \right)=
\alpha.
\end{equation}
Again, $\hat Q(\alpha)$ can be determined by simulation to arbitrary precision. 
The adaptive confidence band is
$\{ (\ell_{\sigma_n}(t),u_{\sigma_n}(t)):\ 0\leq t \leq T\}$, where
\begin{equation}
\ell_{\sigma_n}(t) = {\overline  \lscape}_n(t) - \frac{\hat Q(\alpha) \hat\sigma_n(t)}{\sqrt{n}},\ \ \ 
	u_{\sigma_n}(t) = {\overline  \lscape}_n(t) + \frac{\hat Q(\alpha) \hat\sigma_n(t)}{\sqrt{n}} .
\end{equation}

\begin{theorem}[Adaptive Band]
\label{th::adaptiveBand}
Suppose that $\sigma(t)>c>0$ in an interval $[t_* \, , t^*] \subset [0,T]$, for some constant  $c$. Then
\begin{equation}
    \mathbb{P}\Bigl( \ell_{\sigma_n}(t) \leq \mu(t) \leq u_{\sigma_n}(t)\ {\rm for\ all\ }t \in [t_*\, , t^*] \Bigr) \geq 1-\alpha - 
O\left( \frac{(\log n)^{1/2}}{n^{1/8}} \right).
\end{equation}
\end{theorem}

\iftoggle{journalv}{
   \section{New Functional Summaries}
   \input{othersummaries}
}{ 
   \section{The Weighted Silhouette}
   
The $k$th persistence landscape $\lambda(k,t)$ can be interpreted as
a summary function of the persistence diagrams.  A \emph{summary function}
is a functor that takes a persistence diagram and outputs a real-valued 
continuous function.  If the diagram corresponds to the distance function to a random
set, then we have a probability distribution on the space of summary functions
induced by a probability distribution on the original sample space.
%

The persistence landscape is just one of many functions that could be used
to summarize a persistence diagram.
In this section, we introduce a new family of summary functions
called {\em weighted silhouettes}. 

Consider a persistence diagram with $m$ off diagonal points.
In this formulation, we take the 
weighted average of the triangle functions defined in \eqref{eq:triangle}:
\begin{equation}
 \phi(t) = \frac{\sum_{j=1}^m  w_j \Lambda_j(t)}{\sum_{j=1}^m  w_j}.
\end{equation}
Consider two points of the persistence diagram, representing the pairs
$(b_i,d_i)$ and $(b_j,d_j)$. In general, 
we would like to have $w_j \geq w_i$ whenever
$|d_j-b_j| \geq |d_i - b_i|$.
In particular,
let $\phi(t)$ have weights $ w_j = |d_j - b_j|^p$.

\begin{definition}[Power-Weighted Silhouette]
For every $0 < p \leq \infty$
we define the power-weighted silhouette
$$
\phi^{( p)}(t) = \frac{\sum_{j=1}^m |d_j - b_j|^p \Lambda_j(t)}{\sum_{j=1}^m |d_j - b_j|^p}.
$$
\end{definition}

\begin{figure}[tbh]
\begin{center}
\includegraphics[scale=0.85]{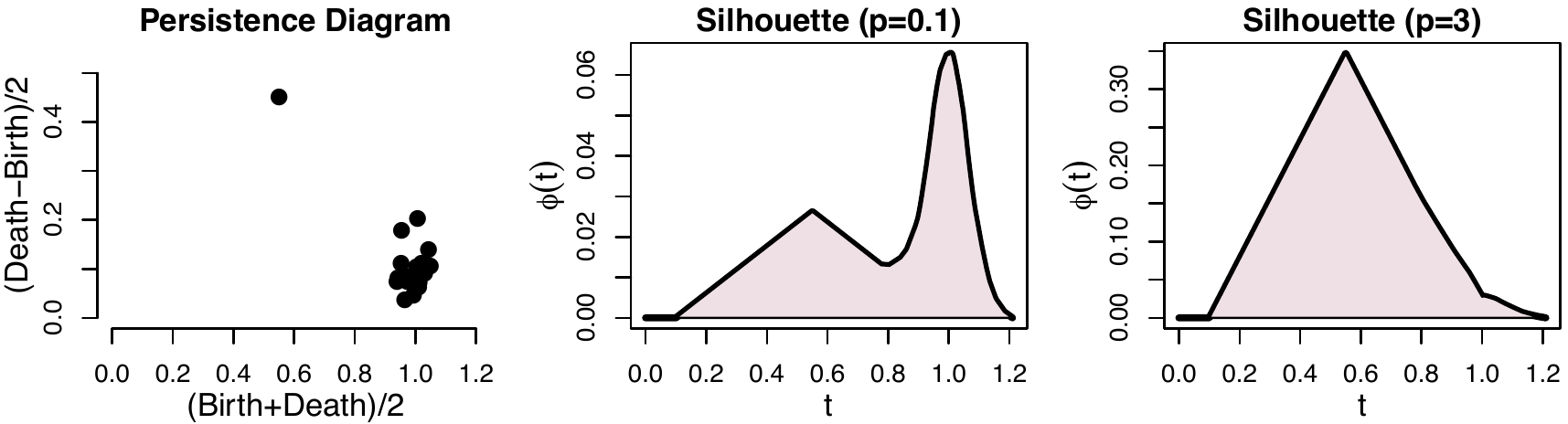}
\end{center}
\caption{An example of power-weighted silhouettes for 
different choices of $p$. Note that the axes are on different scales. The weighted silhouette is one-Lipschitz.}
\label{fig::ExSilhouettes}
\end{figure}
The value $p$ can be though of as a trade-off parameter between
uniformly treating all pairs in the persistence diagram 
and considering only the most persistent pairs.  Specifically, when $p$ is small,
$\phi^{( p)}(t)$ is dominated by the effect of low persistence pairs.
Conversely, when $p$ is large,
$\phi^{( p)}(t)$ is dominated by the most persistent pair;
see Figure \ref{fig::ExSilhouettes}.

The power-weighted silhouette preserves the property of being one-Lipschitz. 
In fact, this is true for any choice of non-negative weights.
Therefore all the result of Sections \ref{ss::weak} and \ref{ss::bootstrap} hold 
for the weighted silhouette, by simply replacing $\lambda$ with $\phi$. 
In particular, consider $\phi_1, \ldots, \phi_n \sim P_\phi$. 
Applying theorems \ref{th::weak}, \ref{th::CLT}, \ref{th::band} 
and \ref{th::adaptiveBand}, we obtain:
\begin{corollary}
  The empirical process 
        $\sqrt{n}\left(n^{-1}\sum_{i=1}^n \phi_i(t) - \mathbb{E}[\phi(t)]\right)$ 
        converges weakly to a Brownian bridge. The rate of 
        convergence of the maximum of this process to the maximum 
        of the limiting distribution is $O\left(\frac{(\log n)^{7/8}}{n^{1/8}} \right)$.
\end{corollary}
\begin{corollary}
  The multiplier bootstrap algorithm of 
        Algorithm \ref{alg::bootstrap} can be used to construct  a uniform 
        confidence band for $\{\mathbb{E}[\phi(t)] \}_{t \in [t_*, t^*]}$ with 
        coverage at least $1-\alpha - O\left( \frac{(\log n)^{7/8}}{n^{1/8}} \right)$ 
        and an adaptive confidence band with coverage at least $1-\alpha - 
  O\left( \frac{(\log n)^{1/2}}{n^{1/8}} \right)$ , where $[t_*,t^*] \subset [0,T]$ 
       is such that $\sqrt{\text{Var}(\phi(t))}>c>0$ for all $t \in [t_*,t^*]$ and some constant $c$. 
\end{corollary}
}

\section{Examples}
\vskip-5pt
In Topological Data Analysis, persistent homology is classically used to encode the evolution of the homology of filtered simplicial complexes built on top of data sampled from a metric space - see \cite{cdso-psvrc-12}.
For example, given a metric space $(\X ,d_{\X})$ and a probability distribution $P_{\X}$ supported on $\X$, one can sample $m$ points, $K = \{ X_1, \ldots, X_m \}$, i.i.d. from $P_{\X}$ and consider the Vietoris-Rips filtration built on top of these points: $\sigma = [X_{i_0}, \ldots, X_{i_k}] \in R(K,a)$ if and only if $d_{\X}(X_{i_j}, X_{i_l}) \leq a$ for any $j,l \in \{0, \ldots k \}$. The persistent homology of this filtration induces a persistent diagram $D$ and a landscape $\lambda$. Sampling $n$ such $K$, one obtains $n$ persistence landscapes $\lscape_1, \ldots, \lscape_n$.
In this section, we adopt this setting to illustrate our results on two examples, one real and one simulated. 

%
%

\subsection{Earthquake data}
\label{section::earthquake}
\begin{figure}[tbh]
\begin{center}
\includegraphics[scale=0.15]{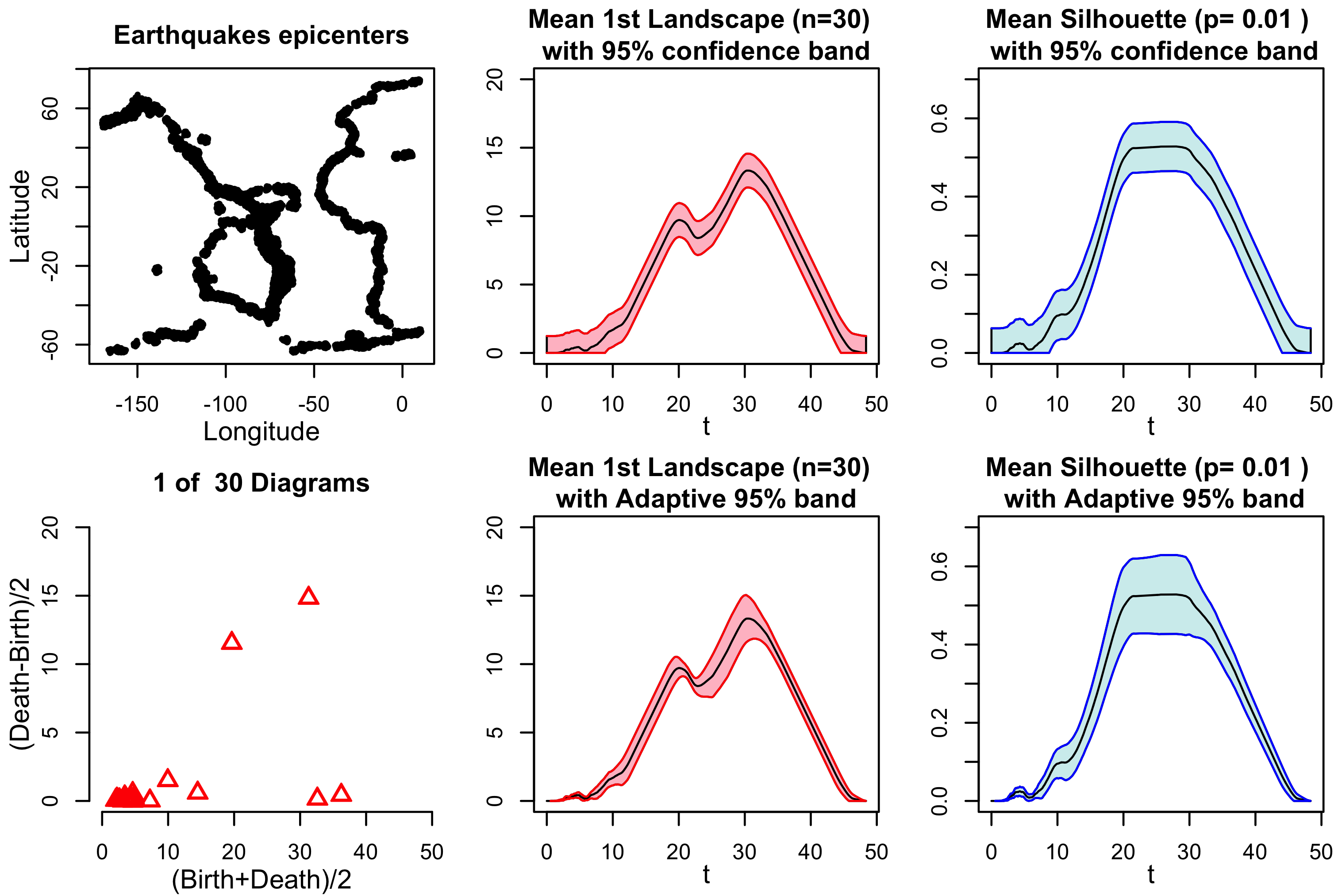}
\end{center}
\caption{Top Left: Sample space of epicenters of 8000 earthquakes. Bottom Left: one of the 30 persistence diagrams. 
        Middle: uniform and adaptive $95\%$ confidence bands for the mean landscape $\mu(t)$. Right: uniform and adaptive 95\% confidence bands for the mean weighted silhouette $\mathbb{E}[\phi^{(0.01)}(t)]$.}
\label{fig::earthquake}
\end{figure}
Figure \ref{fig::earthquake} (left) shows the epicenters of 8000 earthquakes 
in the latitude/longitude rectangle $[-75,75] \times [-170,10]$ of magnitude 
greater than 5.0 recorded between 1970 and 2009.\footnote{USGS Earthquake
Search.   http://earthquake.usgs.gov/earthquakes/search/.}
We randomly sample $m=400$ epicenters, construct the Vietoris-Rips filtration (using the Euclidean distance), 
compute the persistence diagram (Betti 1) using 
Dionysus\footnote{Dionysus is a C++ library for computing persistent homology,
developed by Dmitriy Morozov.  http://mrzv.org/software/dionysus/.} and the corresponding landscape 
function. We repeat this procedure $n=30$ times and compute the mean 
landscape~$\overline{\lscape}_n$. Using the algorithm given in 
Algorithm \ref{alg::bootstrap}, we obtain the uniform 95\% confidence band of 
Theorem \ref{th::band} and the adaptive 95\% confidence band of 
Theorem \ref{th::adaptiveBand}. See 
Figure \ref{fig::earthquake} (middle). Both the confidence bands 
have coverage around $95\%$ for the mean landscape $\mu(t)$ that 
is attached to the distribution induced by the sampling~scheme.
Similarly, using the same $n=30$ persistence diagrams we construct 
the corresponding weighted silhouettes using $p=0.01$ and 
construct uniform and adaptive $95\%$ confidence bands for 
the mean weighted silhouette $\mathbb{E}[\phi^{(0.01)}(t)]$. 
See Figure \ref{fig::earthquake} (right).  Notice that, for
most $t \in [0, T]$, the adaptive confidence band is
tighter than the fixed-width confidence band.

\subsection{Toy Example: Rings}
\begin{figure}[thb]
\begin{center}
\includegraphics[scale=0.15]{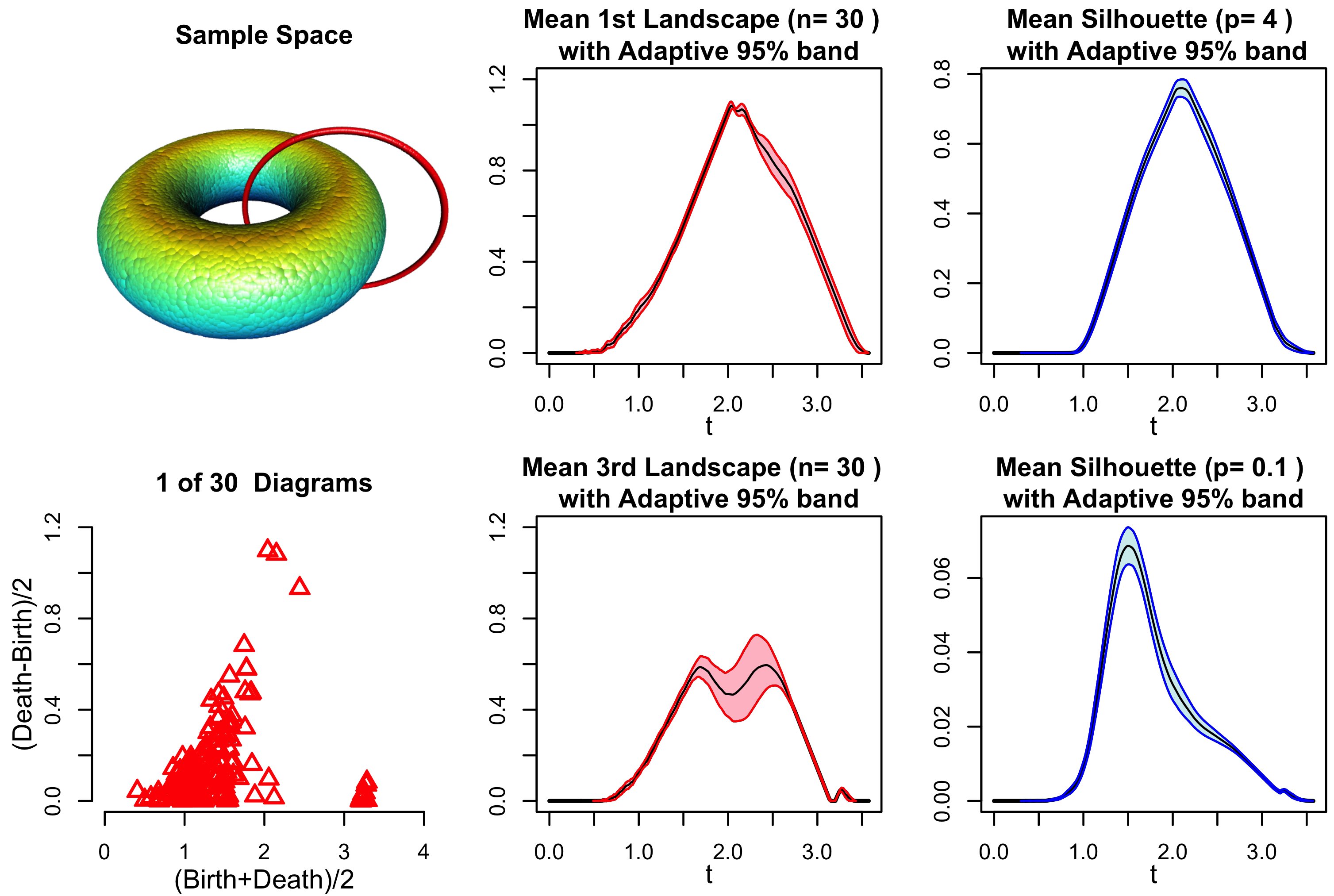}
\end{center}
\caption{Top Left: The sample space. 
         Bottom Left: one of the 30 persistence diagrams. 
         Middle: adaptive 95\% confidence bands for the mean first landscape $\mu_1(t)$ and mean third landscape $\mu_3(t)$.  
         Right: adaptive 95\% confidence bands for the mean weighted silhouettes $\mathbb{E}[\phi^{(4)}(t)]$ and $\mathbb{E}[\phi^{(0.1)}(t)]$.}
\label{fig::rings}
\end{figure}
In this example, we embed the torus $\mathbb{S}^1 \times \mathbb{S}^1$ 
in $\mathbb{R}^3$ and we use the rejection sampling algorithm 
of \cite{diaconis2012sampling} (R = 5, r = 1.8) to sample 10,000 points 
uniformly from the torus. Then we link it with a circle of 
radius 5, from which we sample 1,800 points; see Figure \ref{fig::rings} 
(top left). These $N=11,800$ points constitute the sample space. We 
randomly sample $m=600$ of these points, construct the Vietoris-Rips 
filtration, compute the persistence diagram (Betti 1) and the 
corresponding first and third landscapes and the silhouettes for $p=0.1$ and $p=4$.
We repeat this procedure $n=30$ times to construct 95\% adaptive confidence 
bands for the mean landscapes $\mu_1(t)$, $\mu_3(t)$ and the mean 
silhouettes $\mathbb{E}[\phi^{(4)}(t)]$, $\mathbb{E}[\phi^{(0.1)}(t)]$.
Figure \ref{fig::rings} (bottom left) shows one of the 30 persistence 
diagrams. In the persistence diagram, notice that three persistence pairs
are more persistent than the rest.  These correspond to the two nontrivial
cycles of the torus and the cycle corresponding to the circle.  
We notice that many of the points in the persistence diagram are
hidden by the first landscape.
However, as shown in the figure, the 
third landscape function
and the silhouette with parameter $p=0.1$ are 
able to detect the presence of these features.

\section{Discussion}
\vskip-5pt

We have shown how the bootstrap
can be used to give confidence bands for
Bubeknik's persistence landscape and 
for the persistence silhouette defined in this paper.
We are currently working on several extensions
to our work including the following:
allowing persistence diagrams with countably many points,
allowing $T$ to be unbounded,
and extending our results to new functional summaries of persistence diagrams.
In the case of subsampling (scenario 2 defined in the introduction),
we have provided accurate inferences for the mean function
$\mu$. We are investigating methods to
estimate the difference between $\mu$
(the mean landscape from subsampling)
and $\lambda$
(the landscape from the original large dataset).
Coupled with our confidence bands for $\mu$,
this could provide an efficient approach
to approximating the persistent homology
in cases where exact computations are prohibitive.

\bibliography{paper}

\appendix
\section{Results from \cite{chernozhukov2013anti}}
\label{sec::appendixA}

In this appendix, we summarize the results from \cite{chernozhukov2013anti}
that are used in this paper.
Given a set of functions
${\cal G}$ and a probability measure $Q$, define
the
covering number
$N(\mathcal{G}, L_2(Q), \varepsilon)$
as the smallest number of balls of size $\varepsilon$
needed to cover ${\cal G}$, where the balls are defined
with respect to the norm
$||g||^2 = \int g^2(u) dQ(u)$.
Let $X_1, \dots, X_n$ be i.i.d. random variables taking values in a
measurable space $(S, \mathcal{S})$. Let $\mathcal{G}$ be a class of
functions defined on $S$ and uniformly bounded by a constant $b$, such
that the covering numbers of $\mathcal{G}$ satisfy
\begin{equation}
\label{eq::covering}
\sup_{Q} N(\mathcal{G}, L_2(Q), b\tau) \leq (a/\tau)^v \, , \; 0<\tau <1
\end{equation}
for some $a \geq e$ and $v \geq 1$ and where the supremum is taken over
all probability measures $Q$ on $(S, \mathcal{S})$. The set
$\mathcal{G}$ is said to be of VC type, with constants $a$ and $v$ and envelope
$b$.
Let $\sigma^2$ be a constant
such that $\sup_{g \in \mathcal{G}} E[g(X_i)^2] \leq \sigma^2 \leq
b^2$ and for some sufficiently large constant $C_1$,
denote $K_n:= C_1 v (\log n \vee \log(ab/\sigma))$.  Finally, let
$$
W_n:= \Vert \mathbb{G}_n\Vert_{\mathcal{G}}:= \sup_{g \in \mathcal{G}}
|\mathbb{G}_n(g)|
$$
denote the supremum of the empirical process $\mathbb{G}_n$.

\begin{theorem}[Theorem 3.1 in \cite{chernozhukov2013anti}]
\label{theorem::CCK3.1} Consider the setting specified above.
For any $\gamma \in (0,1)$, there is a random variable $W \stackrel{d}{=} \Vert
\mathbb{G} \Vert_\mathcal{G}$ such that
$$
\mathbb{P}\left( |W_n - W| > \frac{b K_n}{\gamma^{1/2}n^{1/2}}+
\frac{\sigma^{1/2}K_n^{3/4}}{\gamma^{1/2}n^{1/4}}+
\frac{b^{1/3}\sigma^{2/3}K_n^{2/3}}{\gamma^{1/3}n^{1/6}} \right) \leq C_2
\left(\gamma + \frac{\log n}{n} \right)
$$
for some constant $C_2$.
\end{theorem}

Let $\xi_1, \dots, \xi_n$ be i.i.d.\ $N(0,1)$ random variables
independent of $X_1^n:=\{X_1, \dots, X_n \}$. Let
$\xi_1^n:=\{\xi_1, \dots, \xi_n\}$. Define the Gaussian multiplier
process
$$
\tilde{\mathbb{G}}_n(g) = \tilde{\mathbb{G}}_n(X_1^n, \xi_1^n)(g):=
\frac{1}{\sqrt{n}} \sum_{i=1}^n \xi_i \left( g(X_i)- E_n[g(X_i)] \right), \quad
g \in \mathcal{G}.
$$
Lastly, for fixed $x_1^n$, let $\tilde W_n(x_1^n):=\sup_{g \in \mathcal{G}}
|\tilde{\mathbb{G}}_n(x_1^n, \xi_1^n)(g)|$ denote the
supremum of this process.

\begin{theorem}[Theorem 3.2 in \cite{chernozhukov2013anti}]
\label{theorem::CCK3.2}
Consider the setting specified above.
Assume that $b^2 K_n \leq n \sigma^2$. For any $\delta > 0$ there
exists a set $S_{n} \in \mathcal{S}^n$ such that $\mathbb{P}(S_{n}) \geq
1-3/n$ and for any $x_1^n \in S_{n}$ there is a random variable $W
\stackrel{d}{=} \sup_{g \in \mathcal{G}} |\mathbb{G} |$ such that
$$
\mathbb{P}\left( |\tilde{W}_n(x_1^n) - W| > \frac{\sigma K_n^{1/2}}{n^{1/2}}+
  \frac{b^{1/2}\sigma^{1/2}K_n^{3/4}}{n^{1/4}}+\delta \right) \leq C_3
\left( \frac{b^{1/2}\sigma^{1/2}K_n^{3/4}}{\delta n^{1/4}} +
  \frac{1}{n}\right)
$$
for some constant $C_3$.
\end{theorem}

\begin{theorem}[Gaussian anti-concentration, Corollary 2.1 in
\cite{chernozhukov2013anti}]
\label{theorem::anti-concentration}
Let $W=(W_t)_{t\in T}$ be a separable Gaussian process indexed by a semimetric
space $T$ such that $E[W_t]=0$ and $E[W_t^2]=1$ for all $t \in T$. Assume that
$\sup_{t \in T} W_t < \infty$ a.s. Then, $a(|W|):=E[\sup_{t \in T}|W_t|] \in
[\sqrt{2/\pi}, \infty)$ and
$$
\sup_{x \in \mathbb{R}} \mathbb{P}\left(\Big| \sup_{t \in T} |W_t| - x  \Big|
\leq \varepsilon \right) \leq A \varepsilon a(|W|)
$$
for all $\varepsilon \geq 0$ and some constant $A$.
\end{theorem}

\begin{theorem}[Gaussian anti-concentration, Lemma 6.1 in
\cite{chernozhukov2012gaussian}]
\label{theorem::anti-concentrationSigma}
Let $(S,\mathcal{S},P)$ be a probability space, and let $\mathcal{F} \subset
{L}^2( P)$ be a $P$-pre-Gaussian class of functions. Denote by $\mathbb{G}$ a
tight Gaussian random element in $\ell^\infty(\mathcal{F})$ with mean zero and
covariance function $\mathbb{E}[\mathbb{G}(f) \mathbb{G}(g)]=\text{Cov}_P(f,g)$
for all $f,g \in \mathcal{F}$. Suppose that there exist constants
$\underline \sigma$, $\overline \sigma >0$ such that $\underline \sigma^2 \leq
\text{Var}_P(f) \leq \overline \sigma^2$ for all $f \in \mathcal{F}$. Then for
every $\varepsilon>0$,
$$
\sup_{x \in \mathbb{R}} \mathbb{P} \left( \left| \sup_{f \in \mathcal{F}}
\mathbb{G} f-x \right| \leq \varepsilon \right) \leq C_\sigma \varepsilon \left(
\mathbb{E}\left[ \sup_{f \in \mathcal{F}} \mathbb{G} f\right] + \sqrt{1 \vee
\log(\underline \sigma/\epsilon)}\right),
$$
where $C_\sigma$ is a constant depending only on $\underline \sigma$ and
$\overline \sigma$.
\end{theorem}

\begin{theorem}[Talagrand's inequality, Theorem A.4 in
\cite{chernozhukov2013anti}]
\label{th::talagrand}
Let $\xi_1, \dots, \xi_n$ be i.i.d. random variables taking values in a
measurable space $(S,\mathcal{S})$. Suppose that $\mathcal{G}$ is a measurable
class of functions on $S$ uniformly bounded by a constant $b$ such that there
exist constants $a\geq \text{e}$ and \mbox{$v>1$} with $\sup_Q N(\mathcal{G},
L_2(Q), b\varepsilon) \leq (a/\varepsilon)^v$ for all $0\ < \varepsilon < 1$.
Let $\sigma^2$ be a constant such that $\sup_{g \in \mathcal{G}} \text{Var}(g)
\leq \sigma^2 \leq b^2$.
If $b^2 v \log(ab(\sigma) \leq n \sigma^2$, then for all $t \leq n
\sigma^2/b^2$,
$$
\mathbb{P} \left( \sup_{g \in \mathcal{G}} \left | \sum_{i=1}^n \{g(\xi_i) -
\mathbb{E}[g(\xi_1)] \} \right | > A \sqrt{n \sigma^2 \left[ t \vee\left( v \log
\frac{ab}{\sigma}\right) \right]} \right) \leq e^{-t},
$$
where $A$ is an absolute constant.
\end{theorem}

\section{Technical Tools}
\label{sec::appendixB}
In this section, we prove some results that will be used in the proofs of Appendix
C.
Some of our techniques are an adaptation of the strategy used in
\cite{chernozhukov2013anti} to construct adaptive confidence bands.

Consider the class of functions ${\cal F} = \{ f_t  \}_{0\leq t\leq T}$, defined
in \eqref{eq::classF} and let $\lscape_1^n = (\lscape_1,\ldots,\lscape_n)$ be an
i.i.d. sample from a probability $P$ on the measurable space $(\mathcal{L}_T,\mathcal{S})$
of persistence landscapes. We summarize the processes used in the analysis of
persistence landscapes, given in Sections \ref{ss::weak} and
\ref{ss::bootstrap}:
\begin{itemize}
\item
$\mathbb{G}(f_t)$ is a Brownian Bridge with covariance function $$
\kappa(t,u) = \int f_t(\lambda)f_u(\lambda) dP(\lambda) - \int f_t(\lambda)
dP(\lambda) \int f_u(\lambda) dP(\lambda),
$$
\item
$\displaystyle
\mathbb{G}_n(f_t)=\frac{1}{\sqrt{n}}\sum_{i=1}^n (f_t(\lscape_i) - \mu(t)),
$
\item
$\displaystyle
\tilde{\mathbb{G}}_n(f_t) = \tilde{\mathbb{G}}_n \left(\lscape_1^n,
\xi_1^n\right)(f_t)=
\frac{1}{\sqrt{n}} \sum_{i=1}^n \xi_i \left( f_t(\lscape_i)-
\overline{\lscape}_n(t) \right).
$

\end{itemize}
For $\sigma(t)>c>0$, we also defined
\begin{itemize}
\item
$\displaystyle
\mathbb{H}_n(f_t) = \mathbb{H}_n(\lscape_1^n)(f_t):=
\frac{1}{\sqrt{n}} \sum_{i=1}^n  \frac{ f_t(B_i)- \mu(t) }{\sigma(t)},
$

\item
$\displaystyle
\hat{\mathbb{H}}_n(f_t) = \tilde{\mathbb{H}}_n(\lscape_1^n, \xi_1^n)(f_t):=
\frac{1}{\sqrt{n}} \sum_{i=1}^n \xi_i \frac{ f_t(\lscape_i)-
\overline{\lscape}_n(t) }{\hat\sigma_n(t)},
$
\end{itemize}
and for completeness we introduce
\begin{itemize}
\item $\mathbb{H}(f_t)$, the standardized Brownian Bridge with covariance
function
\begin{equation}
\label{eq::standCovariance}
\kappa(t,u) = \int \frac{f_t(\lambda)f_u(\lambda)}{\sigma(t)\sigma(u)}
dP(\lambda) - \int \frac{f_t(\lambda)}{\sigma(t)} dP(\lambda) \int
\frac{f_u(\lambda)}{\sigma(u)} dP(\lambda),
\end{equation}

\item The process
\begin{equation}
\label{eq::standBoot}
\tilde{\mathbb{H}}_n(f_t) := \hat{\mathbb{H}}_n(\lscape_1^n, \xi_1^n)(f_t):=
\frac{1}{\sqrt{n}} \sum_{i=1}^n \xi_i \frac{ f_t(\lscape_i)-
\overline{\lscape}_n(t) }{\sigma(t)},
\end{equation}
which differs from $\hat{\mathbb{H}}_n(f_t)$ in the use of the standard
deviation $\sigma(t)$ that replace its estimate~$\hat \sigma_n(t)$.
\end{itemize}

\begin{proposition}[Supremum Convergence]
\label{prop::bootstrap}
Suppose that $\sigma(t)>c>0$ in an interval $[t_* \, , t^*] \subset [0,T]$, for
some constant  $c$. Then, for large $n$,
there exists a random variable $W\stackrel{d}{=} \sup_{t \in [t_* \, , t^*]}
|\mathbb{G}(f_t) |$ and a set~$S_{n} \in \mathcal{S}^n$ such that $\mathbb{P}(
\lscape_1^n \in S_{n})\geq 1-3/n$
and, for any fixed $\breve \lambda_1^n:=(\breve \lambda_1, \dots, \breve
\lambda_n) \in S_{n}$,
$$
\sup_{z\in \R} \left| \mathbb{P}\left(\sup_{t \in [t_* \, , t^*]}
|\tilde{\mathbb{G}}_n(\breve \lambda_1^n, \xi_1^n)(f_t)| \leq z \right) -
\mathbb{P} \left(W \leq z \right) \right| \leq C_6 \left( \frac{(\log
n)^{5/8}}{n^{1/8}}  \right)
$$
for some constant $C_6>0$.
\end{proposition}
\begin{proof}
Let $\mathcal{F}^*=\{f_t \in \mathcal{F}: t \in [t_* \, , t^*] \}$.
Consider the covering number $N({\cal F^*},L_2(Q),||F||_2 \varepsilon)$ of the
class~$\mathcal{F^*}$, as defined in Appendix \ref{sec::appendixA}, with $F =
T/2$. In the proof of Theorem \ref{th::CLT} we show that
$$
\sup_Q N({\cal F^*},L_2(Q),||F||_2 \varepsilon) \leq 2/\varepsilon,
$$
where the supremum is taken over all measures $Q$ on $\lscapespace$.\\
For $n>2$, $b=\sigma=T/2$, $v=1$, $K_n= A\, (\log n \vee 1)$, Theorem
\ref{theorem::CCK3.2} implies that there exists a set $S_{n}$ such that
$\mathbb{P}(\lscape_1^n \in S_{n})\geq 1-3/n$ and, for any fixed $\breve \lambda_1^n:=(\breve
\lambda_1, \dots, \breve \lambda_n) \in S_{n}$ and $\delta>0$,
$$
\mathbb{P} \left( \big| \sup_{t \in [t_* \, , t^*]} |\mathbb{\tilde G}_n|  - W
\big| > \frac{T (A \log n)^{1/2}}{2 n^{1/2}} + \frac{T (A \log n)^{3/4}}{2
n^{1/4}}+ \delta \right) \leq C_3 \left(\frac{T (A \log n)^{3/4}}{2 \delta
n^{1/4}}+ \frac{1}{n} \right).
$$
Define
$$
g(n,\delta,T):=\frac{T (A \log n)^{1/2}}{2 n^{1/2}} + \frac{T (A \log
n)^{3/4}}{2 n^{1/4}}+ \delta .
$$
Using the strategy of Theorem \ref{th::CLT} and applying the anti-concentration
inequality of Theorem \ref{theorem::anti-concentrationSigma}, it follows that
for large $n$ and $\breve \lambda_1^n:=(\breve \lambda_1, \dots, \breve
\lambda_n) \in S_{n}$,
\begin{equation}
\label{eq::multiBootstrap2}
\sup_z \left | \mathbb{P}\left(\sup_{t \in [t_* \, , t^*]}
|\tilde{\mathbb{G}}_n(\breve \lambda_1^n, \xi_1^n)| \leq z \right) -
\mathbb{P}( W \leq z) \right |\leq  C_5 \; g(n,\delta,T) \sqrt{\log
\frac{c}{g(n,\delta,T)}} +C_3 \left(\frac{T (A \log n)^{3/4}}{2 \delta n^{1/4}}+
\frac{1}{n} \right)
\end{equation}
for some constant $C_5>0$. Choosing $\delta= \frac{(A \log n)^{1/8}}{n^{1/8}}$,
we have
$$
g(n,\delta,T)=\frac{T (A \log n)^{1/2}}{2 n^{1/2}} + \frac{T (A \log n)^{3/4}}{2
n^{1/4}}+ \frac{(A \log n)^{1/8}}{n^{1/8}}.
$$
The result follows by noticing that,
$$
g(n,\delta,T)=O\left(\frac{(\log n)^{1/8}}{n^{1/8}} \right)
$$
and
$$
\sqrt{\log \frac{c}{g(n, \delta, T)}}=O\left( (\log n)^{1/2} \right).
$$
\end{proof}

In the following lemma we consider the class $\mathcal{G}_c=\{g_t: g_t=
f_t/\sigma(t), \;  t_*\leq t \leq t^*  \}$ where $f_t \in \mathcal{F}$ is
defined in \eqref{eq::classF} and we bound the corresponding covering number, as
in \eqref{eq::covering}.
\begin{lemma}
\label{lem::VCtype}
Consider the assumptions of Theorem \ref{th::adaptiveBand} and consider the
class of functions $\mathcal{G}_c=\{g_t: g_t= f_t/\sigma(t), \;  t_*\leq t \leq
t^*  \}$, where $f_t \in \mathcal{F}$. Note that $T/(2c)$ is a measurable
envelope for $\mathcal{G}_c$. Then
$$
\sup_Q N(\mathcal{G}_c, L_2(Q), \varepsilon \Vert T/(2c)  \Vert_{Q,2}) \leq
(a/\varepsilon)^v, \; 0< \varepsilon < 1
$$
for $a=(T^2+2c^2)/c^2 $ and $v=1$, where the supremum is taken over all measures
$Q$ on $\lscapespace$. $\mathcal{G}_c$ is of VC type, with constants $a$ and $v$
and envelope $T/(2c)$.
\end{lemma}
\begin{proof}
First, using the definition of $\sigma(t)$ given in \eqref{eq::sigma}, for $t>u$
we have
\begin{align*}
\sigma^2(t)-\sigma^2(u)&=\text{Var}(f_t(\lambda_1))-\text{Var}(f_u(\lambda_1))\\
&= \mathbb{E}[f_t^2(\lambda_1)]-(\mathbb{E}[f_t(\lambda_1)])^2 -
\mathbb{E}[f_u^2(\lambda_1)] + (\mathbb{E}[f_u(\lambda_1)])^2\\
&= \mathbb{E}[f_t^2(\lambda_1)- f_u^2(\lambda_1)]   +
(\mathbb{E}[f_u(\lambda_1)])^2 -(\mathbb{E}[f_t(\lambda_1)])^2\\
&= \mathbb{E}\left[ \left( f_t(\lambda_1)- f_u(\lambda_1) \right)
\left(f_t(\lambda_1)+ f_u(\lambda_1)\right) \right]   + \\
& \quad \left(\mathbb{E}[f_u(\lambda_1)] -\mathbb{E}[f_t(\lambda_1)] \right)
\left(\mathbb{E}[f_u(\lambda_1)] +\mathbb{E}[f_t(\lambda_1)] \right)\\
&\leq (t-u) \big( \mathbb{E}[ f_t(\lambda_1)+ f_u(\lambda_1)]+ \mathbb{E}[f_u(\lambda_1)]
+\mathbb{E}[f_t(\lambda_1)] \big)\\
&\leq 2(t-u) T.
\end{align*}
Note that we used the fact that $f_t(\lambda)$ is 1-Lipschitz in $t$ and $T/2$
is an envelope of $\mathcal{F}$.
Therefore
$$
| \sigma(t)-\sigma(u) | = \frac{|\sigma^2(t)-\sigma^2(u)|}{\sigma(t)+\sigma(u)}
\leq \frac{|t-u| T}{c}.
$$
Using that $f_t(\lambda)$ is one-Lipschitz, we also have that $|\sigma(t)
g_t(\lambda) - \sigma(u)g(u) | \leq |t-u|$, for $t,u \in [t_*,t^*]$. Construct a
grid $t_* \equiv t_0 < t_1 < \cdots < t_N \equiv t^*$ such that $t_{j+1}-t_j=
\frac{\varepsilon T c^2}{T^2+2c^2}$. We claim that $\{g_{t_j}: 1\leq j \leq N
\}$ is an $\varepsilon T/(2c)$-net of $\mathcal{G}_c$: if $g_t$ in
$\mathcal{G}_c$, then there exists a $j$ so that $t_j \leq t \leq t_{j+1}$ and
\begin{align*}
\Vert g_{t_{j+1}} - g_t \Vert_{Q,2} &= \left \Vert \frac{\sigma(t_{j+1})
g_{t_{j+1}}}{\sigma(t_{j+1})} -\frac{\sigma(t) g_t}{\sigma(t)} \right
\Vert_{Q,2}\\
& = \left \Vert \frac{\sigma(t_{j+1})\sigma(t) g_{t_{j+1}} -
\sigma(t_{j+1})\sigma(t) g_t }{\sigma(t_{j+1}) \sigma(t)}  \right \Vert_{Q,2}\\
&= \left \Vert \frac{\sigma(t_{j+1})\sigma(t) g_{t_{j+1}} - \sigma^2(t_{j+1})
g_{t_{j+1}} + \sigma^2(t_{j+1}) g_{t_{j+1}} - \sigma(t_{j+1})\sigma(t) g_t
}{\sigma(t_{j+1}) \sigma(t)}  \right \Vert_{Q,2}\\
&= \left \Vert \frac{\sigma(t_{j+1})g_{t_{j+1}} [\sigma(t)  - \sigma(t_{j+1}) ]
+ \sigma(t_{j+1}) [\sigma(t_{j+1}) g_{t_{j+1}} - \sigma(t) g_t]
}{\sigma(t_{j+1}) \sigma(t)}  \right \Vert_{Q,2}\\
&\leq \left \Vert \frac{T [\sigma(t)  - \sigma(t_{j+1}) ]  }{2c^2}  \right
\Vert_{Q,2} +  \frac{ t_{j+1}-t}{c}   \\
&\leq \frac{(t_{j+1}-t) T^2}{2 c^3} + \frac{ t_{j+1}-t}{c}\\
&\leq (t_{j+1}-t_j)
\frac{T^2+2c^2}{2c^3}\\
&= \frac{\varepsilon T c^2}{T^2+2c^2} \, \frac{T^2+2c^2}{2c^3}\\
&= \frac{\varepsilon T}{2c}.
\end{align*}
Thus
$$
\sup_Q N(\mathcal{G}_c, L_2(Q), \varepsilon  T/(2c) ) \leq
\frac{(T^2+2c^2)(t^*-t_*)}{\varepsilon T c^2} \leq \frac{T^2+2c^2}{\varepsilon
c^2}.
$$
\end{proof}

Let $\mathbb{H}$ be a Brownian bridge with covariance function given in
\eqref{eq::standCovariance}.
\begin{lemma}
\label{lem::standardCLT}
One can construct a random variable $Y \stackrel{d}{=} \sup_{t \in [t_*,t^*]}
|\mathbb{H}|$ such that for large $n$,
$$
\mathbb{P}\left( \big | \sup_{t \in [t_*,t^*]} | \mathbb{H}_n (f_t)| - Y \big |
> C_7 \frac{(\log n)^{1/2}}{n^{1/8}}\right) \leq C_8 \frac{(\log
n)^{1/2}}{n^{1/8}}.
$$
for some absolute constants $C_7$ and $C_8$.
\end{lemma}
\begin{proof}
The result follows by combining Lemma \ref{lem::VCtype} and Theorem
\ref{theorem::CCK3.1}, with $\gamma=\frac{(\log n)^{1/2}}{n^{1/8}}$.
\end{proof}

Consider $\sigma(t)$ and $\hat \sigma(t)$, defined in \eqref{eq::sigma} and
\eqref{eq::sigmaHat}.
\begin{lemma} For large $n$ and some constant $C_9$,
\label{lem::sigmas}
  \begin{equation}\label{eq::varianceratiointerval}
    \mathbb{P} \left( \sup_{t \in [t_*, t^*]} \left|
    \frac{\hat\sigma_n(t)}{\sigma(t)} -1 \right| \geq C_9 \frac{(\log
    n)^{1/2}}{n^{1/2}} \right) \leq \frac{2}{n}.
  \end{equation}
\end{lemma}
\begin{proof}
Let $\mathcal{G}_c=\{g_t: g_t= f_t/\sigma(t), \;  t_*\leq t \leq t^*  \}$ and
$\mathcal{G}_c^2:=\{g^2: g \in \mathcal{G}_c  \}$.\\
By definition
$\displaystyle \hat \sigma_n^2(t)= \frac{1}{n} \sum_{i=1}^n f^2_t(\lambda_i) -
[\overline{\lscape}_n(t)]^2$ and
$\displaystyle
\sigma^2(t)=\mathbb{E}[f_t^2(\lambda_1)]-(\mathbb{E}[f_t(\lambda_1)])^2 $.
Thus
\begin{align}
\left| \frac{\hat\sigma_n(t)}{\sigma(t)} -1 \right| &\leq \left|
\frac{\hat\sigma^2_n(t)}{\sigma^2(t)} -1 \right| = \left|
\frac{\hat\sigma^2_n(t)- \sigma^2(t)}{\sigma^2(t)} \right| \nonumber \\
& \leq \sup_{t \in [t_*, t^*]} \left | \frac{1}{n} \frac{\sum_{i=1}^n
f_t^2(\lambda_i)}{\sigma^2(t)} -
\frac{\mathbb{E}[f^2_t(\lambda_1)]}{\sigma^2(t)} \right | +  \sup_{t \in [t_*,
t^*]} \left | \left[ \frac{1}{n} \frac{\sum_{i=1}^n f_t(\lambda_i)}{\sigma(t)}
\right]^2 - \left[ \frac{\mathbb{E}[f_t(\lambda_1)]}{\sigma(t)} \right]^2 \right
| \nonumber \\
&= \sup_{g \in \mathcal{G}^2_c} \left | \frac{1}{n} \sum_{i=1}^n g(\lambda) -
\mathbb{E}[g(\lambda)] \right | +
\sup_{g \in \mathcal{G}_c} \left | \left[\frac{1}{n} \sum_{i=1}^n g(\lambda)
\right]^2 - \left(\mathbb{E}[g(\lambda)] \right)^2 \right |
\label{eq::sigmas}
\end{align}

Using the same strategy of Lemma \ref{lem::VCtype}, it can be shown that
$\mathcal{G}_c^2$ is VC type with some constants $A$ and $V \geq 1$ and envelope
$T^2 / (4c^2)$.
Therefore, by Theorem \ref{th::talagrand}, with $t=\log n$ and for large $n$,
\begin{equation}
\label{eq::talagrand1}
\mathbb{P} \left( \sup_{g \in \mathcal{G}^2_c} \left | \frac{1}{n} \sum_{i=1}^n
g(\lambda) - \mathbb{E}[g(\lambda)] \right | > C_{10} \frac{(\log
n)^{1/2}}{n^{1/2}}   \right) \leq \frac{1}{n}.
\end{equation}
Note that
$$
\sup_{g \in \mathcal{G}_c} \left | \left[\frac{1}{n} \sum_{i=1}^n g(\lambda)
\right]^2 - \left(\mathbb{E}[g(\lambda)] \right)^2 \right |
\leq \frac{T}{c}  \sup_{g \in \mathcal{G}_c} \left | \frac{1}{n} \sum_{i=1}^n
g(\lambda)  - \mathbb{E}[g(\lambda)]  \right |
$$
and applying again Theorem \ref{th::talagrand} to the right hand side we obtain
\begin{equation}
\label{eq::talagrand2}
\mathbb{P} \left( \sup_{g \in \mathcal{G}_c} \left | \left[\frac{1}{n}
\sum_{i=1}^n g(\lambda) \right]^2 - \left(\mathbb{E}[g(\lambda)] \right)^2
\right | > C_{11} \frac{(\log n)^{1/2}}{n^{1/2}}   \right) \leq \frac{1}{n}.
\end{equation}

The inequality of \eqref{eq::varianceratiointerval} follows by combining
\eqref{eq::sigmas}, \eqref{eq::talagrand1} and \eqref{eq::talagrand2}.
\end{proof}

\begin{lemma}[Estimation error of $ \hat Q(\alpha)$]
\label{lem::quantile}
Let $Q(\alpha)$ be the $(1-\alpha)$-quantile of the random variable $Y
\stackrel{d}{=}\sup_{t \in [t_*,t^*]} | \mathbb{H}|$ and $\hat Q(\alpha)$ be the
$(1-\alpha)$-quantile of the random variable $\sup_{t \in [t_*,t^*]} |
\hat{\mathbb{H}}_n|$. There exist positive constants $C_{12}$ and $C_{13}$ such
that for large $n$:
\begin{itemize}
\item[(i)] $\displaystyle \mathbb{P}\left[ \hat Q(\alpha) < Q\left(\alpha +
C_{12}\frac{(\log n)^{3/8}}{n^{1/8}}\right) - C_{13}\frac{(\log
n)^{3/8}}{n^{1/8}} \right] \leq \frac{5}{n} $, \nopagebreak
\item[(ii)] $\displaystyle \mathbb{P}\left[ \hat Q(\alpha) > Q\left(\alpha -
C_{12}\frac{(\log n)^{3/8}}{n^{1/8}}\right) + C_{13}\frac{(\log
n)^{3/8}}{n^{1/8}} \right] \leq \frac{5}{n} $.
\end{itemize}
\end{lemma}
\begin{proof}
Define $\Delta \mathbb{H}_n(f_t):= \hat{\mathbb{ H}}_n(f_t) - \mathbb{\tilde
H}_n(f_t)$. Consider the set $S_{n,1} \in \mathcal{S}^n$ of values $\breve \lambda_1^n$ such that,
whenever $\lambda_1^n \in S_{n,1}$,
$$
\left| \frac{\hat \sigma(t)}{\sigma(t)}-1\right| \leq C_9 \frac{(\log
n)^{1/2}}{n^{1/2}} \quad \text{for all } t \in [t_*, t^*].
$$
By Lemma \ref{lem::sigmas}, $\mathbb{P}(\lambda_1^n \in S_{n,1})\geq 1- 2/n$. Fix $\breve
\lambda_1^n \in S_{n,1}$. Then
$$
\Delta \mathbb{H}_n(\breve \lambda_1^n, \xi_1^n)(f_t):=
   \frac{1}{\sqrt{n}} \sum_{i=1}^n \xi_i \frac{f_t(\breve \lambda_i)-
\overline{\lscape}_n(t)}{\sigma(t)} \left(\frac{\sigma(t)}{\hat \sigma_n(t)} -1
\right)
$$
is a zero-mean Gaussian process with variance
$$
\frac{\hat \sigma_n^2(t)}{\sigma^2(t)} \left(\frac{\sigma(t)}{\hat \sigma_n(t)}
-1 \right)^2 \leq C_9^2 \frac{\log n}{n}.
$$
Let $\mathcal{\tilde G}_c=\{a g: a \in (0,1], g \in \mathcal{G}_c \}$.
$\mathcal{\tilde G}_c$ is VC type with some constants $A$ and $V\geq1$ and
envelope $T^2 /(4c^2)$. Moreover, the uniform covering number of the process
$\Delta \mathbb{H}_n(\breve \lambda_1^n, \xi_1^n)(f_t)$ with respect to the
natural semimetric (standard deviation) is bounded by the uniform covering
number of $\mathcal{\tilde G}_c$. Therefore we can apply Theorem 2.4 in
\cite{talagrand1994sharper} (see also Section A.2.2 in \cite{van1996weak})
and obtain
\begin{align}
\mathbb{P}\left( \left| \sup_{t \in [t_*, t^*]} |\hat{\mathbb{H}}(\breve
\lambda_1^n)(f_t)| - \sup_{t \in [t_*, t^*]} |\mathbb{\tilde H}(\breve
\lambda_1^n)(f_t)|  \right| \geq \beta_n \right) &\leq
\mathbb{P} \left( \sup_{t \in [t_*, t^*]} | \Delta \mathbb{H}_n(\breve
\lambda_1^n, \xi_1^n)(f_t) | \geq \beta_n \right) \nonumber \\
& \leq D \left( \frac{\beta_n n}{C_9^2 \log n} \right)^V \frac{C_9\sqrt{\log
n}}{ \beta_n \sqrt{n}} \exp\left(- \frac{\beta_n^2 n}{2C_9^2 \log n} \right),
\end{align}
for some constant $D$. For $C_{14}=\sqrt{2} C_9 (1+V/2)^{1/2}$ and $\beta_n=
C_{14} (\log n)/n^{1/2}$, the last quantity is bounded~by
$$
C_{15} \frac{1}{n (\log n)^{1/2}},
$$
for some constant $C_{15}$.
Therefore, for large $n$,
\begin{align}
&\mathbb{P}\left( \left| \sup_{t \in [t_*, t^*]} |\hat{\mathbb{H}}(\breve
\lambda_1^n)(f_t)| - \sup_{t \in [t_*, t^*]} |\mathbb{\tilde H}(\breve
\lambda_1^n)(f_t)|  \right| \geq C_{14} \frac{(\log n)^{3/8}}{n^{1/8}} \right)
\nonumber \\
&\leq \mathbb{P}\left( \left| \sup_{t \in [t_*, t^*]} |\hat{\mathbb{ H}}(\breve
\lambda_1^n)(f_t)| - \sup_{t \in [t_*, t^*]} |\mathbb{\tilde H}(\breve
\lambda_1^n)(f_t)|  \right| \geq C_{14} \frac{(\log n)}{n^{1/2}} \right)
\nonumber \\
& \leq C_{15} \frac{1}{n (\log n)^{1/2}}  \leq C_{15} \frac{(\log
n)^{3/8}}{n^{1/8}}. \label{eq::sn1}
\end{align}

By Theorem \ref{theorem::CCK3.2} with $\delta= \frac{(\log n)^{3/8}}{n^{1/8}}$,
for large $n$, there exists a set $S_{n,2} \in \mathcal{S}^n$ such that
$\mathbb{P}(\lambda_1^n \in S_{n,2})\geq
1-3/n$, and for any $\breve \lambda_1^n \in S_{n,2}$, one can construct a random
variable $Y \stackrel{d}{=} \sup_{t \in [t_*,t^*]} |\mathbb{H}|$ such that
\begin{equation}
\label{eq::sn2}
\mathbb{P}\left( \left| \sup_{t \in [t_*, t^*]} |\mathbb{\tilde H}(\breve
\lambda_1^n)(f_t)| - Y \right| \geq C_{16} \frac{(\log n)^{3/8}}{n^{1/8}}
\right) \leq C_{17}  \frac{(\log n)^{3/8}}{n^{1/8}}.
\end{equation}
Combining \eqref{eq::sn1} and \eqref{eq::sn2}, we have that, for large $n$ and
$\breve \lambda_1^n \in S_{n,0}:= S_{n,1} \cap S_{n,2}$,
\begin{equation}
\mathbb{P}\left( \left| \sup_{t \in [t_*, t^*]} |\hat{\mathbb{H}}(\breve
\lambda_1^n)(f_t)| - Y \right| \geq C_{13} \frac{(\log n)^{3/8}}{n^{1/8}}\right)
\leq C_{12}  \frac{(\log n)^{3/8}}{n^{1/8}},
\end{equation}
for some constants $C_{12}, C_{13}$.\\
Let $\hat Q(\alpha, \breve \lambda_1^n)$ be the conditional
$(1-\alpha)$-quantile of $\sup_{t \in [t_*, t^*]} | \hat{\mathbb{H}}(\breve
\lambda_1^n)(f_t)|$.
Then $\hat Q(\alpha)= \hat Q(\alpha,\breve  \lambda_1^n)$ is a random quantity and for
$\breve \lambda_1^n \in S_{n,0}$, we have that
\begin{align*}
& \mathbb{P}\left( Y \leq \hat Q(\alpha, \breve \lambda_1^n) +  C_{13}
\frac{(\log n)^{3/8}}{n^{1/8}} \right) \\
&\geq \mathbb{P}\left( \left\{Y \leq \hat Q(\alpha, \breve \lambda_1^n) +
C_{13} \frac{(\log n)^{3/8}}{n^{1/8}} \right\} \bigcap
    \left\{   \left| \sup_{t \in [t_*, t^*]} |\hat{\mathbb{ H}}(\breve
\lambda_1^n)(f_t)| - Y \right| \leq C_{13} \frac{(\log n)^{3/8}}{n^{1/8}}
\right\} \right)\\
& \geq \mathbb{P}\left( \sup_{t \in [t_*, t^*]} | \hat{\mathbb{H}}(\breve
\lambda_1^n)(f_t)| \leq \hat Q(\alpha, \breve \lambda_1^n)  \right) -  C_{12}
\frac{(\log n)^{3/8}}{n^{1/8}}\\
&\geq 1-\alpha -  C_{12} \frac{(\log n)^{3/8}}{n^{1/8}}.
\end{align*}
Therefore $Q\left(\alpha+C_{12} \frac{(\log n)^{3/8}}{n^{1/8}}\right) \leq \hat
Q(\alpha)+ C_{13} \frac{(\log n)^{3/8}}{n^{1/8}}$ whenever $\lambda_1^n \in
S_{n,0}$, which happens with probability at least $1-5/n$. This proves part (i)
of the theorem. The proof of part (ii) is similar and therefore is~omitted.
\end{proof}

\section{Main Proofs}\label{append:proofs}


\begin{proof}[Proof of Theorem \ref{th::CLT}]
Let $\mathcal{F}^*=\{f_t \in \mathcal{F}: t \in [t_* \, , t^*] \}$.
The Lipschitz property implies that for every $\lscape \in \lscapespace$,
$|f_t(\lscape) - f_u(\lscape)| = |\lscape(t)- \lscape(u)| \leq |t-u|$  and hence
$$
\Vert f_t -f_u \Vert_{Q,2} \leq |t-u|.
$$
Construct a grid,
$0\equiv t_0 < t_1 < \cdots < t_N \equiv T$
where $t_{j+1} - t_{j} := \varepsilon \Vert F \Vert_{Q,2} = \varepsilon \, T/2$.
In the last equality, we used the constant envelope $F(\lscape)=T/2$.
We claim that
$\{f_{t_j}: 1\leq j \leq N \}$ is an $(\varepsilon \, T/2)-$net of ${\cal F^*}$:
choosing $f_t \in \cal F^*$, then there exists a $j$ so that $t_j \leq t \leq
t_{j+1}$ and
$$
\Vert f_{t_{j+1}} -  f_t \Vert_{Q,2} \leq |t_{j+1}-t| \leq
|t_{j+1}-t_j|=\varepsilon \, T/2.
$$
Thus, we have a bound for the covering number of $\mathcal{F^*}$, as in
\eqref{eq::covering}:
$$
\sup_Q N({\cal F^*},L_2(Q),||F||_2 \varepsilon) \leq \frac{T}{\varepsilon \Vert
F \Vert_{Q,2}}=  2/\varepsilon,
$$
where the supremum is taken over all measures $Q$ on $\lscapespace$.\\
By Theorem \ref{theorem::CCK3.1}, with $b=\sigma=T/2$, $v=1$, $K_n= A
\, (\log n \vee 1)$, there exists $W \stackrel{d}{=}
\sup_{f \in \mathcal{F}^*}\mathbb{G}$ such that, for $n>2$,
$$
\mathbb{P}\left( \big| \sup_{t \in [t_* \, , t^*]} |\mathbb{G}_n|  - W \big| >
\frac{T A\log n}{2 \gamma^{1/2}n^{1/2}} +  \frac{T^{1/2} (A \log
n)^{3/4}}{\gamma^{1/2} n^{1/4}} + \frac{T (A \log n)^{2/3}}{2
\gamma^{1/3}n^{1/6}} \right) \leq C_2 \left(\gamma + \frac{\log n}{n} \right)
$$
for some constants $C_2$.\\
Let
$$g(n,\gamma,T)=\frac{T A \log n}{2 \gamma^{1/2}n^{1/2}} +  \frac{T^{1/2} (A
\log n)^{3/4}}{\gamma^{1/2} n^{1/4}} + \frac{T (A \log n)^{2/3}}{2
\gamma^{1/3}n^{1/6}} $$
and define the event
$ E:=\left\{ \big| \sup_{t \in [t_* \, , t^*]} |\mathbb{G}_n|  - W \big| > g(n,
\gamma, T) \right\}.$
Then for any $z$ and large $n$,
\begin{align*}
\mathbb{P}\left(\sup_{t \in [t_* \, , t^*]} |\mathbb{G}_n| \leq z \right) -
\mathbb{P}(W \leq z) &= \mathbb{P}\left(\sup_{t \in [t_* \, , t^*]}
|\mathbb{G}_n| \leq z \; , \; E \right) - \mathbb{P}(W \leq z) +
\mathbb{P}\left(\sup_{t \in [t_* \, , t^*]} |\mathbb{G}_n| \leq z \; , \; E^c
\right)\\
& \leq \mathbb{P}\left( W \leq z+ g(n, \gamma, T)  \right) - \mathbb{P}(W \leq
z)+ \mathbb{P}(E^c)\\
& \leq C_4 \, g(n, \gamma, T) \sqrt{\log \frac{c}{g(n, \gamma, T)}} + C_2
\left(\gamma + \frac{\log n}{n} \right),
\end{align*}
where in the last step we used the anti-concentration inequality of Theorem
\ref{theorem::anti-concentrationSigma} .\\
Similarly,
\begin{align*}
\mathbb{P}(W \leq z) - \mathbb{P}\left(\sup_{t \in [t_* \, , t^*]}
|\mathbb{G}_n| \leq z \right) & \leq
\mathbb{P}(W \leq z, E) - \mathbb{P}\left(\sup_{t \in [t_* \, , t^*]}
|\mathbb{G}_n| \leq z, E \right) + P(E^c)\\
& \leq  \mathbb{P}(W \leq z, E) - \mathbb{P}\left(W \leq z - g(n, \gamma, T) , E
\right) + P(E^c)\\
& \leq  \mathbb{P}\left(z- g(n, \gamma, T) \leq W \leq z, E \right)  + P(E^c)\\
& \leq C_4 \, g(n, \gamma, T) \sqrt{\log \frac{c}{g(n, \gamma, T)}}+ C_2
\left(\gamma + \frac{\log n}{n} \right).
\end{align*}
It follows that
\begin{equation}
\label{eq::CLT2}
\sup_z \left | \mathbb{P}\left(\sup_{t \in [t_* \, , t^*]} |\mathbb{G}_n| \leq z
\right) -  \mathbb{P}( W \leq z) \right | \leq C_4 \, g(n, \gamma, T) \sqrt{\log
\frac{c}{g(n, \gamma, T)}}+ C_2 \left(\gamma + \frac{\log n}{n} \right).
\end{equation}
Choosing $\gamma= \frac{(A \log n)^{7/8}}{n^{1/8}}$, we have
$$
g(n,\gamma,T)=\frac{T (A \log n)^{9/16}}{2 n^{7/16}} +  \frac{T^{1/2} (A \log
n)^{5/16}}{n^{3/16}} + \frac{T (A \log n)^{3/8}}{2 n^{1/8}} .
$$
The result follows by noticing that,
$$
g(n,\gamma,T)=O\left(\frac{(\log n)^{3/8}}{n^{1/8}} \right)
$$
and
$$
\sqrt{\log \frac{c}{g(n, \gamma, T)}}=O\left( (\log n)^{1/2} \right).
$$
\end{proof}

\begin{proof}[Proof of Theorem \ref{th::band} (Uniform Band)] \text{} \\
Follows from Theorem \ref{th::CLT} and Proposition \ref{prop::bootstrap}.\\
The second statement follows from the fact that $\tilde{Z}(\alpha)=O_P(1)$,
where $\tilde{Z}(\alpha)$ is defined in \eqref{eq::sim}.
\end{proof}

\begin{proof}[Proof of Theorem \ref{th::adaptiveBand} (Adaptive Band)] \text{}
\\
Let $\mathbb{H}(f_t)$ be the Brownian bridge with covariance function given in
\eqref{eq::standCovariance}.
Consider $Y \stackrel{d}{=} \sup_{t \in [t_*,t^*]} |\mathbb{H}|$.
Let $Q(\alpha)$ be the $(1-\alpha)$-quantile of $Y$ and $\hat Q(\alpha)$ be the
$(1-\alpha)$-quantile of the random variable
$\sup_{t \in [t_*,t^*]} |  \hat{\mathbb{H}}_n|$.
Let $\varepsilon_1(n)=C_7 (\log n)^{1/2}/n^{1/8}$, $\varepsilon_2(n)=C_{13}
(\log n)^{3/8}/n^{1/8}$, $\varepsilon_3(n)=C_9 (\log n)^{1/2}/n^{1/2}$,6 and
define $\varepsilon(n)=\varepsilon_1(n)+\varepsilon_2(n)+\varepsilon_3(n)
Q(\alpha)$.
Similarly let $\delta_1(n)=C_8 (\log n)^{1/2}/n^{1/8}$, $\delta_2(n)=5/n$,
$\delta_3(n)=2/n$, and define $\delta(n)=\delta_1(n)+\delta_2(n)+\delta_3(n)$ .
Define $\tau(n)=C_{12}(\log n)^{3/8}/n^{1/8}$.\\
Then for large $n$,
\begin{align*}
&\mathbb{P}\Bigl( \ell_\sigma(t) \leq \mu(t) \leq u_\sigma(t)\ {\rm for\ all\ }t
\in [t_*\, , t^*] \Bigr)\\
&= \mathbb{P}\left( \sup_{t \in [t_*,t^*]} \left| \mathbb{H}_n(f_t)
\frac{\sigma(t)}{\hat\sigma_n(t)} \right|  \leq \hat Q(\alpha) \right)\\
&\geq \mathbb{P}\left[ \sup_{t \in [t_*,t^*]} \left| \mathbb{H}_n(f_t) \right|
\leq \left(1- \varepsilon_3(n) \right) Q\left(\alpha + \tau(n) \right) -
\varepsilon_2(n) \right]- \delta_2(n)-\delta_3(n),
\end{align*}
where we applied Lemmas \ref{lem::sigmas} and \ref{lem::quantile}.
Using Lemma \ref{lem::standardCLT} the last quantity is no smaller than
\begin{align*}
&\mathbb{P}\big[ Y  \leq \left(1- \varepsilon_3(n) \right) Q\left(\alpha +
\tau(n) \right) - \varepsilon_2(n) - \varepsilon_1(n) \big]- \delta_1(n) -
\delta_2(n)-\delta_3(n)\\
&\geq\mathbb{P}\big[ Y  \leq Q\left(\alpha + \tau(n) \right) - \varepsilon(n)
\big]- \delta(n)\\
&\geq \mathbb{P}\big[ Y  \leq  Q\left(\alpha + \tau(n) \right)  \big] - \sup_{x
\in \mathbb{R}} \mathbb{P}\left(\Big| Y - x  \Big| \leq \varepsilon(n) \right) -
\delta(n)\\
&\geq 1-\alpha- \tau(n) - \delta(n) - \sup_{x \in \mathbb{R}} \mathbb{P}\left(\Big| Y - x  \Big| \leq \varepsilon(n) \right)\\
&\geq 1-\alpha- \tau(n) - \delta(n) - A \varepsilon(n),
\end{align*}
where in the last step we applied the anti-concentration inequality of Theorem \ref{theorem::anti-concentration}.
\end{proof}

\end{document}